\documentclass{amsart}
\usepackage{amsmath}
\usepackage{amsthm}
\usepackage{amssymb}
\usepackage{mathrsfs}
\usepackage{cite}
\usepackage{amscd}
\usepackage{fancyhdr}
\usepackage[all]{xy}
\theoremstyle{definition}
\newtheorem{thm}{Theorem}[section]
\newtheorem{definition}[thm]{Definition}
\newtheorem{prop}[thm]{Proposition}
\newtheorem{lem}[thm]{Lemma}
\newtheorem{rem}[thm]{Remark}
\newtheorem{cor}[thm]{Corollary}

\newtheorem{prob}[thm]{Problem}

\newtheorem*{ack}{Acknowledgement}
\numberwithin{equation}{section}
\usepackage[colorlinks,linkcolor=blue,citecolor=blue,urlcolor=red
]{hyperref}

\newcommand{\Z}{\mathbb{Z}}

\newcommand{\Spec}{\operatorname{Spec}}

\newcommand{\Hom}{{\rm Hom}}

\newcommand{\Frac}{\operatorname{Frac}}

\newcommand{\im}{{\rm Im}}
\newcommand{\Cok}{{\rm Coker}}

\newcommand{\Deg}{{\rm deg}}

\newcommand{\Gal}{{\rm Gal}}
\newcommand{\sO}{\mathscr{O}}

\newcommand{\calH}{\mathcal{H}}

\newcommand{\bbP}{\mathbb{P}}

\newcommand{\Def}{\overset{{\rm def}}{=}}
\newcommand{\et}{\text{{\rm \'et}}}

\newcommand{\fppf}{{\rm fppf}}
\newcommand{\Zar}{{\rm Zar}}
\newcommand{\tame}{{\rm tame}}
\newcommand{\rmt}{{\rm t}}
\newcommand{\ct}{\mathrm{ct}}
\newcommand{\G}{\mathbb{G}}
\newcommand{\F}{\mathbb{F}}
\newcommand{\A}{\mathbb{A}}

\newcommand{\ur}{\mathrm{ur}}

\newcommand{\CH}{\mathrm{CH}}

\newcommand{\DIV}{\mathrm{Val}^{div}}

\newcommand{\Cor}{\mathrm{Cor}}
\newcommand{\Res}{\mathrm{Res}}
\newcommand{\dlog}{\mathrm{dlog}}

\begin{document}

\title[Mod $p$ unramified cohomology]{On the mod $p$ unramified cohomology of varieties having universally trivial Chow group of zero-cycles}

\author[S. Otabe]{Shusuke Otabe}
\address{Department of Mathematics, School of Engineering, Tokyo Denki University\\ 5 Senju Asahi\\ Adachi\\ Tokyo 120-8551\\ Japan
}
\email{shusuke.otabe@mail.dendai.ac.jp}

\date{\today}


\begin{abstract}
Auel--Bigazzi--B\"ohning--Graf von Bothmer proved that if a proper smooth variety $X$ over a field $k$ of characteristic $p>0$ has universally trivial Chow group of $0$-cycles, the cohomological Brauer group of $X$ is universally trivial as well. In this paper, we generalize their argument to arbitrary unramified mod $p$ \'etale motivic cohomology groups. We also see that the properness assumption on the variety $X$ can be dropped off by using the Suslin homology together with a certain tame subgroup of the unramified cohomology group. 
\end{abstract}

\maketitle

\thispagestyle{empty}

\section{Introduction}

The goal of the present paper is to give a positive answer (see Corollary \ref{cor-int:main} below) to the following problem posed by Auel et al.

\begin{prob}(cf.\ \cite[Problem 1.2]{ABBB})\label{prob:ABBB} 
Let $X$ be a proper smooth variety over a field $k$ of characteristic $p>0$. Suppose that $X$ has universally trivial Chow group of $0$-cycles, i.e.\ the degree map of the Chow group of zero-cycles is an isomorphism $\Deg\colon\CH_0(X_K)\xrightarrow{\simeq}\Z$ for any field extension $K/k$. Then, is the natural homomorphism $H^{i}(k,\Z/p\Z(j))\to H^{i}_{\ur}(k(X)/k,\Z/p\Z(j))$ an isomorphism for any integers $i,j\ge 0$\,? 
\end{prob}

Note that the formulation is slightly different from theirs. Here, the cohomology group $H^{i}(k,\Z/p\Z(j))$ is the mod $p$ \'etale motivic cohomology group of weight $j$, i.e.\ $H^{i}(k,\Z/p\Z(j))=H^{i-j}_{\et}(k,\Omega^j_{\log})$, and the group $H^{i}_{\ur}(k(X)/k,\Z/p\Z(j))$ is the \textit{unramified cohomology group} of the function field $k(X)$, which is defined as the subgroup of the group $H^{i}(k(X),\Z/p\Z(j))$ consisting of cohomology classes which are unramified at every geometric rank one discrete valuation on $k(X)/k$~(cf.\ \cite[\S 5]{BM13}). As the $p$-cohomological dimension of a field of characteristic $p>0$ is less than or equal to one, we have $H^i_{\ur}(k(X)/k,\Z/p\Z(j))=0$ for $i-j\neq 0,1$. Therefore, the problem is nontrivial only in the case when $i=j$ or $i=j+1$. In the former case, the groups $H^{i}(K,\Z/p\Z(i))$ are naturally isomorphic to the mod $p$ Milnor K-groups $K_i^{\rm M}(K)/p$, i.e.\ 
\begin{equation*}
H^i(K,\Z/p\Z(i))\simeq K^{\rm M}_i(K)/p
\end{equation*}
for all field extensions $K/k$ and for all integers $i\ge 0$~(cf.\ \cite[Theorem 2.1]{BK}), and they form a cycle module in the sense of Rost\cite{Rost}. Therefore, by Merkurjev's theorem \cite[Theorem 2.11]{Mer08}, Problem \ref{prob:ABBB} has an affirmative answer in that  case. 
The remaining case is when $i=j+1$. In \cite{ABBB}, Auel et al.\ solved the problem affirmatively for $(i,j)=(2,1)$~(cf.\ \cite[Theorem 1.1]{ABBB}), in which case the unramified cohomology $H^2_{\ur}(k(X)/k,\Z/p\Z(1))$ can be identified with the $p$-torsion subgroup of the Brauer group ${\rm Br}(X)=H^2_{\et}(X,\G_m)$.    

In the present paper, we will extend their argument to the unramified cohomology group $H^{i+1}_{\ur}(k(X)/k,\Z/p\Z(i))$, where $i$ is an arbitrary non-negative integer. As the main result, we will prove the following result.

\begin{thm}(cf.\ Corollary \ref{cor:pairing})\label{thm-int:main}
Let $X$ be a smooth geometrically connected variety over a field $k$ of characteristic $p>0$. Suppose that the degree map $\Deg\colon H^S_0(X_K)\to\Z$ is an isomorphism for any finitely generated field extension $K/k$. Then for any $i\ge 0$, we have a natural isomorphism $H^{i+1}(k,\Z/p\Z(i))\xrightarrow{\simeq} H^{i+1}_{\ct,\ur}(X/k,\Z/p\Z(i))$. 
\end{thm}

Here, $H^S_0(X_K)$ stands for the $0$-th Suslin homology group~(cf.\ \cite[\S3]{Kahn11}) and the group $H^{i+1}_{\ct,\ur}(X/k,\Z/p\Z(i))$ is a certain subgroup of the unramified cohomology group $H^{i+1}_{\ur}(X,\Z/p\Z(i))$, which we call the \textit{unramified curve-tame cohomology group}~(cf.\ Definition \ref{def:ct, ur coh}). In the case where $X$ is proper over $k$, the unramified curve-tame cohomology recovers the unramified cohomology of the function field $k(X)$, i.e.\ $H^{i+1}_{\ct,\ur}(X/k,\Z/p\Z(i))=H^{i+1}_{\ur}(k(X)/k,\Z/p\Z(i))$
, and the natural quotient map $H^S_0(X_K)\twoheadrightarrow\CH_0(X_K)$ is an isomorphism. Therefore, as a corollary of the theorem, we obtain the following result, which gives a positive answer to Problem \ref{prob:ABBB}.

\begin{cor}(cf.\ Corollary \ref{cor:pairing proper})\label{cor-int:main}
Let $X$ be a proper smooth variety over a field $k$ of characteristic $p>0$. Suppose that the degree map $\Deg\colon \CH_0(X_K)\to\Z$ is an isomorphism for any field extension $K/k$. Then for any $i\ge 0$, we have a natural isomorphism $H^{i+1}(k,\Z/p\Z(i))\xrightarrow{\simeq} H^{i+1}_{\ur}(k(X)/k,\Z/p\Z(i))$. 
\end{cor}

After writing up the first version of the present paper, the author learned the paper~\cite{BRS} by Binda--R\"ulling--Saito, in which Corollary \ref{cor-int:main} is obtained as a general fact on   \textit{reciprocity sheaves}. On the other hand, our proof is independent of new framework developed there. In \cite{KOY}, a further different type of approach is discussed.

We explain the organization of the present paper. 
In \S\ref{sec:log HW}, we recall general facts on  the logarithmic Hodge--Witt sheaves. We recall the statement of the Gersten-type conjecture established by Gros--Suwa\cite{Gros-Suwa} and Shiho\cite{Shiho07}(cf.\ Theorem \ref{thm:Gersten}). We also recall basic properties of corestriction map on the mod $p$ \'etale motivic cohomology, which was defined by Kato~(cf.\ \cite{Kato}). 
In \S\ref{sec:ur}, we recall the notion of unramified cohomology and discuss some properties of it.  

In \S\ref{sec:ur tame}, we introduce two kinds of  tame subgroups of the mod $p$ unramified cohomology, namely the \textit{na\"ive unramified tame cohomology} $H^{i+1}_{\tame,\ur}(X/k,\Z/p\Z(i))$~(cf.\ Definition \ref{def:t, ur coh}) and the \textit{unramified curve-tame cohomology} $H^{i+1}_{\ct,\ur}(X/k,\Z/p\Z(i))$~(cf.\ Definition \ref{def:ct, ur coh}). We see the former one admits a corestriction map for any finite surjective morphism of normal varieties~(cf.\ Proposition \ref{prop:Cor on t, ur coh}). However, we cannot see that it has enough functoriality property. For that reason, we consider the latter tame subgroup, which is respected by morphisms between regular varieties. In the case where $X=C$ is a normal curve, these tame subgroups coincide with each other~(cf.\ Proposition \ref{prop:t, ur coh}(1)). The idea of considering tame subgroups goes back to the works due to Kato\cite{Kato82}, Izhboldin\cite{Izhboldin}, Garibaldi--Merkurjev--Serre\cite{GMS}, Auel et al.\ \cite{ABBB tame} and Totaro\cite{Totaro20}. 

In \S\ref{sec:pairing}, we prove Theorem \ref{thm-int:main}. The technical issue is the same as in \cite{ABBB}. Namely, for a smooth variety over a field $k$ of characteristic $p>0$, we construct a family of pairings
\begin{equation*}
\left\{H^S_0(X_K)\times H^{i+1}_{\ct,\ur}(X_K/K,\Z/p\Z(i))\to H^{i+1}(K,\Z/p\Z(i))\right\}_{K}
\end{equation*}
which fulfills a satisfactory compatibility condition, where $K$ is taken over all finitely generated field  extensions of $k$~(cf.\ Theorem \ref{thm:pairing}). To this end, we follow the argument in \cite[\S3]{ABBB}. The idea of dropping off the properness assumption  from the original problem~(cf.\ Problem \ref{prob:ABBB}) goes back to the work of Bruno Kahn\cite{Kahn11}. He generalized Merkurjev's theorem\cite[Theorem 2.11]{Mer08} to an open variety by replacing the Chow group with the Suslin homology group~(cf.\ \cite[Corollary 4.7]{Kahn11}).

\begin{ack}
The author would like to thank Tomoyuki Abe for having fruitful discussions and giving helpful comments. The author is grateful to referees for giving comments and suggestions. The author is supported by JSPS KAKENHI Grant (JP19J00366, JP21K20334). 
\end{ack}

\section*{Notation}

For an equidimensional scheme $X$ and any integer $i\ge 0$, we denote by $X^{(i)}$ (respectively $X_{(i)}$) the set of points of $X$ of codimension $i$ (respectively of dimension $i$).

Let $k$ be a field. A \textit{variety} over $k$ is an integral separated scheme of finite type over $k$. A \textit{curve} over $k$ (or $k$-\textit{curve}) is a variety over $k$ of dimension one. Let $K$ be a finitely generated field over $k$. A \textit{model} of $K/k$ is a proper variety over $k$ together with an isomorphism $k(X)\xrightarrow{\simeq}K$ of fields over $k$. 

Let $K/k$ be a finitely generated field extension of $k$. A valuation $v$ on $K$ over $k$ is a valuation on $K$ such that the associated valuation ring $\sO_v$ contains $k$ as a subalgebra. A discrete rank one valuation $v$ on $K$ over $k$ is said to be \textit{geometric} if it satisfies the condition that
\begin{equation*}
{\rm tr.deg}_k(K)={\rm tr.deg}_k(k(v))+1,
\end{equation*}
where $k(v)$ is the residue field of $v$ and ${\rm tr.deg}_k(L)$ means the transcendental degree over $k$ for any field extension $L/k$. According to \cite[Proposition 1.7]{Mer08}, a discrete rank one valuation $v$ on $K$ over $k$ is geometric if and only if there exists a normal model $X$ of $K/k$ such that the point $x$ dominated by $v$ is of codimension one and $\sO_v=\sO_{X,x}$. A geometric discrete rank one valuation on $K/k$ is also called a \textit{divisorial valuation}~(cf.\ \cite[Definition 2.4]{ABBB}).


\section{The logarithmic Hodge--Witt sheaves}\label{sec:log HW}

Let $X$ be a scheme over the prime field $\F_p$ of positive characteristic $p>0$. For any integer $n\ge 1$, let $W_n\Omega^{\bullet}_{X}$ denote the de Rham--Witt complex of $X/\F_p$~(cf.\ \cite[I, 1.3]{Ill}). Recall that for any morphism of $\F_p$-schemes $f\colon Y\to X$, we have a canonical morphism of complexes of $W_n(\sO_Y)$-modules,
\begin{equation}\label{eq:HW}
f^{-1}W_n\Omega_X^{\bullet}\to W_n\Omega_Y^{\bullet}
\end{equation}
(cf.\ \cite[I, (1.12.3)]{Ill}), which is an isomorphism if $f$ is \'etale~(cf.\ \cite[I, Proposition 1.14]{Ill}).

For any $i\ge 0$, we denote by $W_n\Omega^i_{X,\log}$ the logarithmic Hodge--Witt sheaf of $X$ in the sense of \cite[Definition 2.6]{Shiho07}. Namely it is the \'etale sheaf on $X$ defined as the image
\begin{equation*}
W_n\Omega^i_{X,\log}\Def \im\bigl((\sO_X^{\times})^{\otimes i}\to W_n\Omega^i_X\bigl),
\end{equation*}
of the map $(\sO_X^{\times})^{\otimes i}\to W_n\Omega^i_X\,;\,x_1\otimes\cdots\otimes x_i\mapsto \dlog[x_1]\wedge\cdots\wedge \dlog [x_i]$, where $[x_i]\in W_n\sO_X$ is the Teichm\"uller representative of $x_i$. 
If $f\colon Y\to X$ is a morphism of $\F_p$-schemes, by the functoriality of the de Rham--Witt complexes~(\ref{eq:HW}), there exists a natural morphism of \'etale sheaves on $Y$, 
\begin{equation}\label{eq:log HW}
f^{-1}W_n\Omega_{X,\log}^i\to W_n\Omega^i_{Y,\log}.
\end{equation}
In particular, we have the restriction maps 
$H^{j}_{\et}(X,W_n\Omega^i_{X,\log})\to H^j_{\et}(Y,W_n\Omega_{Y,\log}^i)$. 

According to \cite[Proposition 2.8]{Shiho07}, for a regular scheme $X$ over $\F_p$, we have the exact sequence on $X_{\et}$,
\begin{equation}\label{eq:log HW 1-F}
0\to \Omega^i_{X,\log}\to \Omega^i_X\xrightarrow{1-F}\Omega^i_X/d\Omega^{i-1}_X\to 0,
\end{equation} 
where $F\colon \Omega^i_X\to \Omega^i_X/d\Omega_X^{i-1}$ is the map induced by the Frobenius $F\colon W_{2}\Omega^{\bullet}_X\to W_1\Omega^{\bullet}_X=\Omega^{\bullet}_X$ of the de Rham--Witt complexes~(cf.\ \cite[Lemma 2.7]{Shiho07}).

For an equidimensional scheme $X$ over $\F_p$, we have the coniveau spectral sequence
\begin{equation}\label{eq:coniveau}
E_1^{s,t}=\bigoplus_{x\in X^{(s)}}H^{s+t}_x(X,W_n\Omega^i_{X,\log})\Rightarrow E^{s+t}=H^{s+t}(X,W_n\Omega^i_{X,\log})
\end{equation}
(cf.\ \cite[\S 4]{Shiho07}). We set $B_n^{t,i}(X)^{\bullet}\Def E_1^{\bullet,t}$. Then the  following Gersten-type  conjecture is established by Gros--Suwa for localizations of smooth algebras of finite type over a perfect field of characteristic $p>0$ and by Shiho in the arbitrary case.

\begin{thm}(cf.\ \cite{Gros-Suwa}\cite[Theorem 4.1]{Shiho07})\label{thm:Gersten}
Let $X=\Spec A$ be the spectrum of an equidimensional regular local ring $A$ over $\F_p$. Then we have
\begin{equation*}
H^m(B_{n}^{q,i}(X)^{\bullet})=
\begin{cases}
H^q(X,W_n\Omega^i_{X,\log})& m=0,\\
0& m>0.
\end{cases}
\end{equation*}
\end{thm}

As a consequence, we have the following. For smooth varieties, see also \cite[Proposition A.10]{BM13}.

\begin{cor}(Gersten Type Conjecture)\label{cor:Gersten}
Let $X$ be an equidimensional regular scheme over $\F_p$. Let $\calH^{q,i}(n)_X\Def R^q\varepsilon_*W_n\Omega^{i}_{X,\log}$, where $\varepsilon\colon X_{\et}\to X_{\Zar}$ is the natural map of sites. Then we have an exact sequence
\begin{equation*}
0\to H^0(X_{\Zar},\calH^{q,i}(n)_X)\to\bigoplus_{x\in X^{(0)}}H^q_x(X,W_n\Omega^i_{X,\log})\to\bigoplus_{x\in X^{(1)}}H^{q+1}_x(X,W_n\Omega^i_{X,\log}). 
\end{equation*}
\end{cor}

As a consequence, the cohomology functor $H^{q}(-,W_n\Omega^i_{\log})$ has the injectivity property and the codimension one purity property in \cite[Definition 2.1.4]{CT} as follows.

\begin{cor}\label{cor:inj, codim one purity}~
\begin{enumerate}
\item (Injectivity Property) If $R$ is a regular local ring over $\F_p$ with the field of fractions $K$, then the restriction map 
\[
H^{q}(R,W_n\Omega^i_{R,\log})\to H^{q}(K,W_n\Omega^i_{K,\log})
\] 
is injective. 

\item (Codimension One Purity) If $R$ is a regular local ring over $\F_p$ with the field of fractions $K$, then 
\begin{equation*}
\begin{aligned}
&\im\left(H^q(R,W_n\Omega^i_{R,\log})\to H^q(K,W_n\Omega_{K,\log}^i)\right)\\
=&\bigcap_{{\rm ht}\,\mathfrak{p}=1}\im\left(H^q(R_{\mathfrak{p}},W_n\Omega^i_{R_{\mathfrak{p}},\log})\to H^q(K,W_n\Omega^i_{K,\log})\right)
\end{aligned}
\end{equation*}
as subgroups of $H^q(K,W_n\Omega^i_{K,\log})$.
\end{enumerate}
\end{cor}

Next we will recall some basic properties of the corestriction map
\begin{equation*}
\Cor_{L/K}\colon H^i(L,W_n\Omega^j_{L,\log})\to H^i(K,W_n\Omega^j_{K,\log})
\end{equation*}
defined by Kato~(cf.\ \cite{Kato}), where $L/K$ is a finite extension of fields of characteristic $p>0$. Let $K$ be a field of characteristic $p>0$. Then as $K$ has $p$-cohomological dimension $\le 1$, the group $H^{q}_{\et}(K,W_n\Omega^i_{K,\log})$ is zero unless $q=0,1$. If $q=0$, there exists a natural isomorphism
\begin{equation*}
H^0_{\et}(K,W_n\Omega^i_{K,\log})\simeq K^{\rm M}_i(K)/p^n
\end{equation*}
(cf.\ \cite{BK}). Thus, for any finite extension $L/K$ of fields of characteristic $p>0$, the norm map $N_{L/K}\colon K^{\rm M}_i(L)\to K^{\rm M}_i(K)$ of the Milnor K-groups induces the corestriction map 
$H^0_{\et}(L,W_n\Omega^i_{L,\log})\to H^0_{\et}(K,W_n\Omega^i_{K,\log})$.

Let us consider the case $q=1$. In \cite[p.\ 658]{Kato}, Kato defined the corestriction map
\begin{equation*}
\Cor_{L/K}\colon H^1_{\et}(L,W_n\Omega^i_{L,\log})\to H^1_{\et}(K,W_n\Omega^i_{K,\log})
\end{equation*}
for any finite extension $L/K$ of fields of characteristic $p>0$. The map $\Cor_{L/K}$ is defined by using the norm maps of Quillen's K-groups $K_*(L)\to K_*(K)$~(cf.\ \cite{Quillen}). Recall also that the graded abelian group $\bigoplus_{i\ge 0}H^1_{\et}(K,W_n\Omega^i_{K,\log})$ has a natural right $\bigoplus_{i\ge 0} K^{\rm M}_i(K)/p^n$-module structure~(cf.\ \cite[p.\ 658]{Kato}). We denote by 
\begin{equation*}
[-,-\}\colon H^1_{\et}(K,W_n\Omega^i_{K,\log})\times K^{\rm M}_j(K)/p^n\to H^1_{\et}(K,W_n\Omega^{i+j}_{K,\log})
\end{equation*}
the corresponding multiplication. 

The following are some properties of the map $\Cor_{L/K}$ which we will use later as $n=1$.

\begin{enumerate}
\renewcommand{\labelenumi}{(C\arabic{enumi})}

\item\label{eq:Cor trans} (cf.\ \cite[Remark 5.1(2)]{Shiho07}) For any finite extensions $K''/K'/K$ of fields of characteristic $p>0$, one has
\begin{equation*}
\Cor_{K'/K}\circ\Cor_{K''/K'}=\Cor_{K''/K}.
\end{equation*}

\item\label{eq:proj formula 1}
 (cf.\ \cite[p.\ 658, Lemma 1(1)]{Kato}) Let $L/K$ be a finite extension of fields of characteristic $p>0$. Then for any $w\in\bigoplus_{i\ge 0}H^1_{\et}(L,W_n\Omega^i_{L,\log})$ and any $a\in\bigoplus_{i\ge 0}K^{\rm M}_i(K)/p^n$, one has
\begin{equation*}
\Cor_{L/K}([w,a_L\})=[\Cor_{L/K}(w),a\}.
\end{equation*}

\item\label{eq:proj formula 2} (cf.\ \cite[p.\ 658, Lemma 1(2)]{Kato}). Let $L/K$ be a finite extension of fields of characteristic $p>0$. Then for any $w\in\bigoplus_{i\ge 0}H^1_{\et}(K,W_n\Omega^i_{K,\log})$ and any $a\in\bigoplus_{i\ge 0}K^{\rm M}_i(L)/p^n$, one has
\begin{equation*}
\Cor_{L/K}([w_L,a\})=[w,N_{L/K}(a)\}.
\end{equation*}
In particular, one has
\begin{equation}\label{eq:Cor Res =deg}
\Cor_{L/K}\circ\Res_{L/K}=[L:K].
\end{equation}

\item\label{eq:Res vs Cor} Let $L/K$ be a finite extension of fields of characteristic $p>0$. Let $K'/K$ be an arbitrary field extension. Suppose that $L\otimes_K K'\simeq\prod_i L'_i$ a finite product of fields $L'_i$. Then we have
\begin{equation*}
\Res_{K'/K}\circ \Cor_{L/K}=\sum_i\Cor_{L'_i/K'}\circ \Res_{L_i'/L}.
\end{equation*}
By definition, this can be deduced from the corresponding property of the norm maps of Quillen's K-groups~(cf.\ \cite[Lemma 2.4]{Bloch}).    
\end{enumerate}

\begin{rem}\label{rem:div=>excellent}
The condition that $L\otimes_K K'\simeq\prod_{i}L_i'$ in the last property holds if $L/K$ or $K'/K$ is a separable extension. For example, if $K=k(X)$ is the function field of a normal variety $X$ over a field $k$ and if $x\in X^{(1)}$, then the local ring $\sO_{X,x}$ is an excellent discrete valuation ring~(cf.\ \cite[Lemma 07QU]{SP}) and hence the completion $K\hookrightarrow K_{x}\Def\Frac(\widehat{\sO}_{X,x})$ is a separable extension. Therefore, for any finite extension $L/K$, the tensor product $L\otimes_K K_x$ is isomorphic to a direct product of fields.   
\end{rem}


\section{The unramified cohomology}\label{sec:ur}

In this section, following \cite[\S2--\S4]{CT}, we recall the notion of \textit{unramified cohomology} and discuss several properties of it. Fix a base field $k$ with characteristic $p\ge 0$, a prime number $\ell$ (possibly $\ell=p$) and integers $n>0$ and $j$. Let us consider the \'etale sheaf (complex) on $k$-schemes
\[
\Z/\ell^n\Z(j)=
\begin{cases}
\mu_{\ell^n}^{\otimes j},&\text{$\ell\neq p$},\\
W_n\Omega_{\log}^j[-j],&\text{$\ell=p$},
\end{cases}
\]
where the integer $j$ is assumed to be non-negative when $\ell=p$. For any $k$-scheme $X$ and any non-negative integer $i$, we denote by $H^{i,j}(X)$  the $i$-th \'etale cohomology group with coefficients in $\Z/\ell^n\Z(j)$, i.e.\\
\begin{equation*}
H^{i,j}(X)\Def H^i_{\et}(X,\Z/\ell^n\Z(j)).
\end{equation*}
Note that when $\ell=p$, we have
$H^{i,j}(X)=H^{i-j}_{\et}(X,W_n\Omega^j_{X,\log})$.

Now we define three types of unramified cohomology  associated with the cohomology functor $H^{i,j}(-)$~(see Definitions \ref{def:un coh K/k}, \ref{def:ur coh X} and \ref{def:third ur coh}). It turns out that their value at an integral $k$-variety $X$ coincide with each other when $X$ is proper and smooth over $k$~(see Proposition \ref{prop:ur coh cmp}).

\begin{definition}\label{def:DIV}
Let $K/k$ a finitely generated field extension of $k$. We denote by $\DIV(K/k)$ the set of geometric discrete rank one valuations on $K$ over $k$. 
\end{definition}

\begin{definition}(cf.\ \cite[\S5]{BM13})\label{def:un coh K/k}
For a finitely generated field extension $K/k$, we define the \textit{unramified cohomology} $H^{i,j}_{\ur}(K/k)$ to be
\begin{equation*}
H^{i,j}_{\ur}(K/k)\Def \bigcap_{v\in\DIV(K/k)} \im\left(H^{i,j}(\sO_v)\to H^{i,j}(K)\right), 
\end{equation*}
where $\sO_v$ is the valuation ring associated with each $v$. 
\end{definition}

\begin{rem}
Note that for any $v\in\DIV(K/k)$, the valuation ring $\sO_v$ is a discrete valuation ring with the fraction field $K$. Thus, when $\ell=p$, by Corollary \ref{cor:inj, codim one purity}\,(1), the map $H^{i,j}(\sO_v)\to H^{i,j}(K)$ is injective and the group $H^{i,j}(\sO_v)$ can be naturally viewed as a subgroup of $H^{i,j}(K)$, in which case we have the equality 
\[
H^{i,j}_{\ur}(K/k)=\bigcap_{v\in\DIV (K/k)}H^{i,j}(\sO_v).
\] 
as subgroups of $H^{i,j}(K)$.
\end{rem}

\begin{rem}(Functriality Property for field extensions)
Let $L/K$ be a field extension of finitely generated fields over $k$. Then for any $w\in\DIV(K/k)$, the restriction $v\Def w|_K$ is trivial or it belongs to $\DIV(K/k)$~(cf.\ \cite[Proposition 1.4]{Mer08}). This implies that the natural restriction map $H^{i,j}(K)\to H^{i,j}(L)$ induces a map between the unramified cohomology groups $H^{i,j}_{\ur}(K/k)\to H^{i,j}_{\ur}(L/k)$ and the correspondence $K\mapsto H^{i,j}_{\ur}(K/k)$ is a covariant functor of the category of finitely generated field extensions of $k$ into the category of abelian groups.  
\end{rem}

The second kind unramified cohomology is defined as below.

\begin{definition}\label{def:ur coh X}
For a normal variety $X$ over a field $k$, we define the \textit{unramified cohomology} $H^{i,j}_{\ur}(X)$  to be
\begin{equation*}
H^{i,j}_{\ur}(X)\Def\bigcap_{x\in X^{(1)}}\im\left(H^{i,j}(\sO_{X,x})\to H^{i,j}(k(X))\right). 
\end{equation*}
\end{definition}

\begin{rem}
If $X$ is normal, then as the local ring $\sO_{X,x}$ is a discrete valuation ring with $k\subset\sO_{X,x}$ and with fraction field $k(X)$ for any codimension one point $x\in X^{(1)}$, we have an obvious inclusion
\[
H^{i,j}_{\ur}(k(X)/k)\subseteq H^{i,j}_{\ur}(X).
\] 
as subgroups of $H^{i,j}(k(X))$.
\end{rem}

Finally let us introduce the third version of the unramified cohomology associated with the cohomology functor $H^{i,j}(-)$. 

\begin{definition}\label{def:third ur coh}
For a $k$-scheme $X$, let $\varepsilon\colon X_{\et}\to X_{\Zar}$ be the natural map of sites. Then for any $i,j\ge 0$, we set $\calH^{i,j}_X\Def R^i\varepsilon_*(\Z/\ell^n\Z(j)_X)$ and consider the group
\[
H^0_{\Zar}(X,\calH^{i,j}_X)
\]
of global sections of the Zariski sheaf $\calH_X^{i,j}$ as a version of unramified cohomology attached to the cohomology group $H^{i,j}(-)$.
\end{definition}

\begin{rem}(Full functoriality property, cf.\ \cite[Remark 4.1.2]{CT})\label{rem:full functoriality}
For any morphism $f\colon X\to Y$ of schemes over $k$, there exists a natural map of \'etale sheaves $f^{-1}(\Z/\ell^n\Z(j)_Y)\to \Z/\ell^n\Z(j)_X$, 
which induces a morphism of Zariski sheaves 
\[
f^*\colon f^{-1}\calH^{i,j}_Y\to\calH^{i,j}_X
\] 
in a canonical way. In particular, we have a natural restriction map
\begin{equation}
f^*\colon H^{0}_{\Zar}(Y,\calH^{i,j}_Y)\to H^0_{\Zar}(X,\calH^{i,j}_X) 
\end{equation}
and the correspondence $X\mapsto H^0_{\Zar}(X,\calH^{i,j}_X)$ is a contravariant functor of the category of all $k$-varieties to the category of abelian groups.
\end{rem}

Let $X$ be a normal variety over $k$. 
The restriction map 
\begin{equation*}
H^0_{\Zar}(X,\calH^{i,j}_X)\to H^0_{\Zar}(k(X),\calH^{i,j}_{k(X)})=H^{i,j}(k(X))
\end{equation*} 
factors through the second type unramified cohomology group of $X$~(cf.\ Definition \ref{def:ur coh X}), i.e.
\begin{equation*}
H^0_{\Zar}(X,\calH^{i,j}_X)\to H^{i,j}_{\ur}(X)\to H^{i,j}(k(X)). 
\end{equation*}

Now we can compare the three types of unramified cohomology in the following way. 

\begin{prop}(cf.\ \cite[Theorem 4.1.1]{CT})\label{prop:ur coh cmp}
Let $X$ be a smooth variety over $k$. Then there exists a natural isomorphism of abelian groups
\begin{equation*}
H^0_{\Zar}(X,\calH^{i,j}_X)\xrightarrow{\simeq}H^{i,j}_{\ur}(X).
\end{equation*}
If $X$ is proper over $k$ in addition, then we have 
\begin{equation*}
H^0_{\Zar}(X,\calH^{i,j}_X)\xrightarrow{\simeq}H^{i,j}_{\ur}(X)=H^{i,j}_{\ur}(k(X)/k).
\end{equation*}
Furthermore, in the case when $\ell=p$, the same conclusions hold for regular varieties over $k$.
\end{prop}

\begin{proof}
In the case when $\ell\neq p$, all the statements are included in \cite[Theorem 4.1.1]{CT}. In the case when $\ell=p$, the same proof as there works by replacing the codimension one purity theorem \cite[Theorem 3.8.2]{CT} with our version for logarithmic Hodge-Witt cohomology (see Corollary \ref{cor:inj, codim one purity}\,(2)) and by applying the argument given in \cite[Proposition 2.1.8]{CT}. The last statement is thanks to the fact that the codimension one purity  theorem for $\ell=p$ holds for arbitrary regular local rings over $\F_p$.
\end{proof}
 
As a consequence, we have the following.

\begin{cor}\label{cor:ur coh cmp}~
\begin{enumerate}
\item (Birational Invariance) With the above notation, the unramified cohomology $H^{i,j}_{\ur}(X)$ is a $k$-birational invariant for proper smooth $k$-varieties.  
\item (Functoriality Property) The correspondence $X\mapsto H^{i,j}_{\ur}(X)$ is a contravariant functor of the category of smooth $k$-varieties into the category of abelian groups. Furthermore, in the case when $\ell=p$, the same functoriality property can be extended to the category of regular integral schemes over $k$. 
\end{enumerate}
\end{cor}

\begin{proof}~
\begin{enumerate}
\item As mentioned in \cite[Proposition 2.1.8\,(e)]{CT}, the claim is immediate from the equivalence between Definitions \ref{def:un coh K/k} and \ref{def:ur coh X} for proper smooth (or regular) $k$-varieties in Proposition \ref{prop:ur coh cmp}.

\item As discussed in \cite[Remark 4.1.2]{CT}, the claim follows from the equivalence between Definitions \ref{def:ur coh X} and \ref{def:third ur coh} for  smooth (or regular) $k$-varieties in Proposition \ref{prop:ur coh cmp} together with Remark \ref{rem:full functoriality}. In the case when $\ell=p$, this follows also from the injectivity property (Corollary \ref{cor:inj, codim one purity}\,(1)) by applying the argument in \cite[Proposition 2.1.10]{CT}.
\end{enumerate}
\end{proof}

\begin{rem}\label{rem:ur coh cmp dim=1}
Let us assume that $\ell=p$. Let $C$ be a normal $k$-curve. Then for any $i,j\ge 0$, the restriction map $H^{i,j}(C)\to H^{i,j}(k(C))$ induces a surjective homomorphism 
$H^{i,j}(C)\twoheadrightarrow H^{i,j}_{\ur}(C)$. 
Indeed, as $C$ is of dimension one, this follows from the coniveau spectral sequence 
$E_1^{s,t}=\bigoplus_{x\in C^{(s)}}H^{s+t,,j}_x(C)\Rightarrow E^{s+t}=H^{s+t,j}(C)$  
together with Proposition \ref{prop:ur coh cmp}. 
Therefore, if the restriction map $H^{i,j}(C)\to H^{i,j}(k(C))$ is injective, we have an isomorphism $H^{i,j}(C)\xrightarrow{\simeq}H^{i,j}_{\ur}(C)$.  
\end{rem}

\begin{rem}\label{rem:A^0}
If $\ell\neq p$, then the cohomology group $H^{i,j}(-)=H^i_{\et}(-,\mu_{\ell^n}^{\otimes j})$ gives rise to a Rost's cycle module $M_*$ over $k$ as below
\[
M_*(-)\Def\bigoplus_{*\ge -i}H^{i+*}(-,\mu_{\ell^n}^{\otimes j+*})
\]
(see \cite[Remarks 1.11 and 2.5]{Rost}), and for any smooth variety $X$ over $k$, the unramified cohomology $H^{i,j}_{\ur}(X)$ is canonically identified with the Chow group $A^0(X,M_0)$ with coefficients in $M_0$ in the sense of Rost\,(cf.\ \cite[\S5]{Rost}), i.e.
\[
H^{i,j}_{\ur}(X)\simeq A^0(X,M_0).
\] 

In the case where $\ell=p$, the same description of unramified cohomology in terms of cycle modules still holds when $i=j$. Namely, for the cohomology group $H^{i,i}(-)=H^0(-,W_n\Omega^i_{\log})$, we have
\[
H^{i,i}_{\ur}(X)\simeq A^0(X,K_i^{\rm M}/p^n).
\]
However, as remarked in the first paragraph in the next section, Rost's theory of cycle modules cannot be directly adapted for the unramified cohomology group $H^{i+1,i}_{\ur}(X)$. 
\end{rem}


\section{The unramified curve-tame cohomology}\label{sec:ur tame}

Let us restrict our attention to the mod $p$ \'etale motivic cohomology group
\begin{equation*}
H^{i+1,i}(X)=H^{i+1}_{\et}(X,\Z/p\Z(i))=H^1_{\et}(X,\Omega^i_{X,\log}).
\end{equation*}
In the case $i=0$, we have $H^{1,0}(X)=H^1_{\et}(X,\Z/p\Z)$. If $i=1$, then the group $H^{2,1}(X)$ can be identified with $H^2_{\fppf}(X,\mu_p)$. Let $k$ be a field of characteristic $p>0$ and let us consider the affine line $X=\A^1_k$. For these cases $i=0,1$, by purity, we obtain
\begin{equation*}
H^{i+1,i}(\A^1_k)\xrightarrow{\simeq}H^{i+1,i}_{\ur}(\A^1_k) 
\end{equation*}
(Remark \ref{rem:ur coh cmp dim=1}). As is well-known, the natural map $H^{i+1,i}(k)\to H^{i+1,i}(\A^1_k)$ is far from isomorphic (in general) in both the cases $i=0,1$. Therefore, the unramified cohomology $H^{i+1,i}_{\ur}(X)$ is not $\A^1$-homotopy invariant. Hence, the collection of $\Z$-graded abelian groups $\left\{\bigoplus_{i\ge 0}H^{i+1,i}(K)\right\}_{K}$ does not form a cycle module in the sense of Rost\cite{Rost}. To remedy the situation, we will consider  \textit{tame} subgroups of $H^{i+1,i}(K)$~(cf.\ \cite{Kato82}\cite{Izhboldin}\cite{GMS}\cite{ABBB tame}\cite{Totaro20}).

We begin with the local case. For a complete discrete valuation field $K$ of characteristic $p>0$, we define the \textit{tame cohomology group} $H^{i+1,i}_{\tame}(K)$ to be the kernel of the restriction map $H^{i+1,i}(K)\to H^{i+1,i}(K^t)$, where $K^t$ is a maximal tamely ramified extension of $K$~(cf.\ \cite[\S4]{Totaro20}). Note that as $H^{i+1,i}(K)$ is $p$-torsion, the tame cohomology group $H^{i+1,i}_{\tame}(K)$ can be written also as the kernel of the restriction homomorphism $H^{i+1,i}(K)\to H^{i+1,i}(K^{ur})$, where $K^{ur}$ is the maximal unramified extension of $K$~(cf.\ \cite[Remark 3.7]{ABBB tame}).

For a field $K$ of characteristic $p>0$ with a geometric discrete rank one valuation $v$, we define the \textit{tame cohomology group} $H^{i+1,i}_{\tame, v}(K)\subset H^{i+1,i}(K)$ to be the inverse image of the tame cohomology $H^{i+1,i}_{\tame}(K_v)$ via the restriction map $H^{i+1,i}(K)\to H^{i+1,i}(K_v)$, where $K_v$ is the completion of $K$ with respect to the valuation $v$. Then, according to \cite[Theorem 4.3]{Totaro20}~(see also \cite[\S 2, Corollary 2.7]{Izhboldin} for complete discrete valuation fields), there exists a homomorphism $\partial_v\colon H^{i+1,i}_{\tame,v}(K)\to H^{i,i-1}(k(v))$ called the \textit{residue map} at $v$ which fits into the short exact sequence
\begin{equation}\label{eq:residue map t}
0\to H^{i+1,i}(\sO_v)\to H^{i+1,i}_{\tame,v}(K)\xrightarrow{\partial_v} H^{i,i-1}(k(v))\to 0,
\end{equation}
where $\sO_v$ is the valuation ring of $v$.

The residue map $\partial_v$ can be described as follows. Recall that we have an exact sequence of $\F_p$-vector spaces~(cf.\ (\ref{eq:log HW 1-F})), 
\begin{equation*}
0\to H^{i,i}(K)\to\Omega^{i}_K\xrightarrow{F-1}\Omega^i_{K}/d\Omega_K^{i-1}\to H^{i+1,i}(K)\to 0.
\end{equation*}
We denote by $[f,g_1,\dots,g_i\}$ the image in $H^{i+1,i}(K)$ of the differential form 
\[
f\,\dfrac{dg_1}{g_1}\wedge\cdots\wedge\dfrac{dg_i}{g_i}~\in~\Omega^{i}_{K}.
\] 
Then the tame cohomology group $H^{i+1,i}_{\tame,v}(K)$ is generated by the elements of the form
\begin{equation*}
[ f, g_1,\dots,g_i\}
\end{equation*}
where $f\in K,g_1,\dots,g_i\in K^{*}$ satisfying $v(f)\ge 0$~(cf.\ \cite[Theorem 4.3]{Totaro20}). The residue map $\partial_v\colon H^{i+1,i}_{\tame,v}(K)\to H^{i,i-1}(k(v))$ is now uniquely characterized by the following formula
\begin{equation*}
\partial_v([f, g_1,\dots,g_i\}
)=
\begin{cases}
[\overline{f},\overline{g}_2,\dots,\overline{g}_i\}&\text{if $v(g_1)=1$ and $v(g_i)=0$ for $i\neq 1$},\\
0&\text{if $v(g_i)=0$ for any $i$},
\end{cases}
\end{equation*}
where $\overline{f}$ and $\overline{g_i}$ mean the images of $f$ and $g_i$ respectively in the residue field $k(v)$. It follows immediately from this characterization that for any generator $[f,g_1,\dots,g_i\}\in H^{i+1,i}_{\tame,v}(K)$, we have
\begin{equation}\label{eq:partial_v vs partial_v^M}
\partial_v([f,g_1,\dots,g_i\})=[\overline{f},\partial_v^{\rm M}(\{g_1,\dots,g_i\})\},
\end{equation}
where $\partial_v^{\rm M}$ is the tame symbol of the Milnor $K$-group~(cf.\ \cite[Proposition 7.1.4]{GS06}).

\begin{lem}\label{lem:s_v dvr}
Let $K$ be a field of positive characteristic $p>0$ and $v$ a geometric discrete rank one valuation on $K$. Let $\sO_v$ be the valuation ring and $k(v)$ the residue field. Fix a uniformizer $\pi\in\sO_v$. Then the natural reduction map of cohomology groups $H^{i+1,i}(\sO_v)\to H^{i+1,i}(k(v))$ coincides with the composition of maps
\begin{equation*}
H^{i+1,i}(\sO_v)\xrightarrow{[-,\pi\}}H^{i+2,i+1}_{\tame,v}(K)\xrightarrow{\partial_v}H^{i+1,i}(k(v))
\end{equation*}
up to multiplication by $(-1)^i$.
\end{lem}

\begin{proof}
Noticing that the subgroup $H^{i+1,i}(\sO_v)$ is generated by the elements $[f,g_1,\dots,g_i\}$ with $f\in\sO_v$ and $g_1,\dots,g_i\in\sO_v^{*}$, 
the lemma is immediate from the description of the residue map (\ref{eq:partial_v vs partial_v^M}). 
\end{proof}

\begin{definition}\label{def:tame coh K/k}
Let $k$ be a field of characteristic $p>0$ and $K$ a finitely generated field extension of $k$. We define the \textit{tame cohomology} $H^{i+1,i}_{\tame}(K/k)$ to be
\begin{equation*}
H^{i+1,i}_{\tame}(K/k)\Def\bigcap_{v\in\DIV(K/k)}H^{i+1,i}_{\tame,v}(K)\subset H^{i+1,i}(K).
\end{equation*} 
\end{definition}

Note that for any $v\in\DIV(K/k)$, the graded subspace $\bigoplus_{i\ge 0}H^{i+1,i}_{\tame,v}(K)\subset \bigoplus_{i\ge 0}H^{i+1,i}(K)$ is stable under the multiplication map
\begin{equation*}
[-,-\}\colon H^{i+1,i}(K)\times K_j^{\rm M}(K)/p\to H^{i+j+1,i+j}(K)
\end{equation*}
(cf.\ \cite[p.\ 153]{GMS}), and hence $\bigoplus_{i\ge 0}H^{i+1,i}_{\tame,v}(K)$ is a graded $\bigoplus_{i\ge 0}K^{\rm M}_i(K)/p$-submodule of $\bigoplus_{i\ge 0}H^{i+1,i}(K)$. Therefore, the graded submodule $\bigoplus_{i\ge 0}H^{i+1,i}_{\tame}(K/k)$ is also a graded $\bigoplus_{i\ge 0}K^{\rm M}_i(K)/p$-submodule of the group $\bigoplus_{i\ge 0}H^{i+1,i}(K)$.

It is immediate from the exact sequence (\ref{eq:residue map t}) that there exists a short exact sequence
\begin{equation}\label{eq:tame vs ur}
0\to H^{i+1,i}_{\ur}(K/k)\to H^{i+1,i}_{\tame}(K/k)\xrightarrow{(\partial_{v})}\bigoplus_{v\in\DIV(K/k)}H^{i,i-1}(k(v)).
\end{equation}

Moreover, we have the following.

\begin{thm}(cf.\ \cite[Theorem 4.4]{Totaro20})\label{thm:mod p Faddeev}
Let $k$ be a field of characteristic $p>0$. For any $i\ge 0$, there exists an exact sequence of $\F_p$-vector spaces, 
\begin{equation*}
0\to H^{i+1,i}(k)\to H^{i+1,i}_{\tame}(k(t)/k)\xrightarrow{(\partial_x)_x}\bigoplus_{x\in\bbP^1_{k(0)}}H^{i,i-1}(k(x))\xrightarrow{\sum_x\Cor_{k(x)/k}} H^{i,i-1}(k)\to 0. 
\end{equation*}
\end{thm}

In other words, the tame cohomology satisfies the  the \textit{homotopy property} for $\A^1$ and the \textit{reciprocity} for the projective line $\bbP^1$ in \cite[\S2]{Rost}. 

\begin{rem}\label{rem:H, RC}~
\begin{enumerate}
\item (Homotopy Property for $\A^1$) The short exact sequence
\begin{equation*}
0\to H^{i+1,i}(k)\to H^{i+1,i}_{\tame}(k(t)/k)\xrightarrow{(\partial_x)_x}\bigoplus_{x\in\A^1_{k(0)}}H^{i,i-1}(k(x))\to 0, 
\end{equation*}
is exact.
 
\item (Reciprocity for $\bbP^1$) The sequence
\begin{equation*}
H^{i+1,i}_{\tame}(k(t)/k)\xrightarrow{(\partial_x)_x}\bigoplus_{x\in\bbP^1_{k(0)}}H^{i,i-1}(k(x))\xrightarrow{\sum_x\Cor_{k(x)/k}} H^{i,i-1}(k)
\end{equation*}
is a complex. 
\end{enumerate}
\end{rem}

\begin{rem}(Reciprocity for Curves) More generally, for any $k$-curve $C$, it turns out that the sequence
\[
H^{i+1,i}_{\tame}(k(C)/k)\xrightarrow{(\partial_x)_x}\bigoplus_{x\in C_{(0)}}H^{i,i-1}(k(x))\xrightarrow{\sum_x\Cor_{k(x)/k}} H^{i,i-1}(k)
\]
becomes a complex. Indeed, by taking a finite morphism $C\to\bbP^1_k$ and by making use of corestriction maps on tame cohomology which we discuss later (see (\ref{eq:Cor^t})), one can deduce the claim from Remark \ref{rem:H, RC}\,(2). 
\end{rem}

\begin{lem}(Functoriality Property for field extensions)\label{lem:tame coh fld-ext}
Let $k$ be a field of characteristic $p>0$ and $L/K$ an extension of finitely generated fields over $k$. Then the restriction map $H^{i+1,i}(K)\to H^{i+1,i}(L)$ induces a map $H^{i+1,i}_{\tame}(K/k)\to H^{i+1,i}_{\tame}(L/k)$ and the correspondence $K\mapsto H^{i+1,i}_{\tame}(K/k)$ is a covariant functor of the category of finitely generated field extensions over $k$ into the category of abelian groups. 
\end{lem}

\begin{proof}
Let $\alpha\in H^{i+1,i}_{\tame}(K/k)$ be an element. Suppose that $\alpha_L\not\in H^{i+1,i}_{\tame}(L/k)$. Then by definition, there exists a geometric discrete rank one valuation $w$ on $L$ such that $\alpha_{L_w}\not\in H^{i+1,i}_{\tame}(L_w)$. As $H^{i+1,i}(\sO_w)\subset H^{i+1,i}_{\tame}(L_w)$, the restriction $v\Def w|_K$ must be nontrivial, and hence it gives a geometric discrete rank one valuation on $K$ over $k$~(cf.\ \cite[Proposition 1.4]{Mer08}). Let us consider the commutative diagram of field extensions
\begin{equation*}
\begin{xy}
\xymatrix{
L\ar[r]&L_w\ar[r]&L_w^t\\
K\ar[u]\ar[r]&K_v\ar[u]\ar[r]&K_v^t.\ar[u]
}
\end{xy}
\end{equation*}
Then the condition that $\alpha_{L_w}\not\in H^{i+1,i}_{\tame}(L_w)$ implies that $\alpha_{K_v}\not\in H^{i+1,i}_{\tame}(K_v)$, which is a contradiction. This completes the proof. 
\end{proof}

Next we introduce two types of tame subgroup (see Definitions \ref{def:t, ur coh} and \ref{def:ct, ur coh} below) of the unramified cohomology $H_{\ur}^{i+1,i}(X)$~(cf.\ Definition \ref{def:ur coh X}).

\begin{definition}\label{def:t, ur coh}
Let $k$ be a field of characteristic $p>0$ and $X$ a normal variety over $k$. We define the \textit{na\"ive unramified tame cohomology} $H^{i+1,i}_{\tame,\ur}(X/k)$ to be
\begin{equation*}
H^{i+1,i}_{\tame,\ur}(X/k)\Def H^{i+1,i}_{\ur}(X)\cap H^{i+1,i}_{\tame}(k(X)/k)\subset H^{i+1,i}(k(X)). 
\end{equation*}
\end{definition}

\begin{rem}(Functoriality Property for dominant morphisms)\label{rem:tame, ur dominant map}
Let $f\colon Y\to X$ be a morphism between regular varieties over $k$. Then, by Corollary \ref{cor:ur coh cmp}\,(2), we get a natural restriction map between the unramified cohomology groups $f^*\colon H^{i+1,i}_{\ur}(X)\to H^{i+1,i}_{\ur}(Y)$. If in addition $f$ is dominant, by Lemma \ref{lem:tame coh fld-ext}, we also find that the map $H^{i+1,i}_{\tame}(k(X)/k)\subset H^{i+1,i}(k(X))\xrightarrow{f^*}H^{i+1,i}(k(Y))$ factors through the tame cohomology $H^{i+1,i}_{\tame}(k(Y)/k)$. By putting together these facts, we obtain a natural restriction map between the na\"ive unramified tame cohomology groups 
\[
f^*\colon H^{i+1,i}_{\tame,\ur}(X/k)\to H^{i+1,i}_{\tame,\ur}(Y/k).
\] 
Moreover, the same argument as above implies that for any smooth geometrically connected variety $X$ over $k$ and any finitely generated field extension $K/k$ such that the base change $X_K$ of $X$ along the extension $K/k$ is regular and integral, the projection map $X_K\to X$ induces a map $H^{i+1,i}_{\tame,\ur}(X/k)\to H^{i+1,i}_{\tame,\ur}(X_K/K)$.
\end{rem}

The following is another type of tame subgroup, whose  definition is motivated by the work of Kerz--Schmidt~\cite{KS}. 

\begin{definition}\label{def:ct, ur coh}
Let $k$ be a field of characteristic $p>0$ and $X$ a regular variety over $k$. We define the \textit{unramified curve-tame cohomology}  $H^{i+1,i}_{\ct,\ur}(X/k)$ to be the subgroup of the unramified cohomology group $H^{i+1,i}_{\ur}(X)$ which consists of elements $\alpha\in H^{i+1,i}_{\ur}(X)$ such that for any finitely generated field extension $K/k$ and for any $k$-morphism $C\to X$ from any normal $K$-curve $C$, the restriction $\alpha|_C\in H^{i+1,i}_{\ur}(C)$ belongs to the subgroup $H^{i+1,i}_{\tame,\ur}(C/K)$. 
\end{definition}

\begin{rem}
In the above definition, as $X$ is regular, thanks to Corollary \ref{cor:ur coh cmp}\,(2), the restriction map
\[
H^{i+1,i}_{\ur}(X)\to H^{i+1,i}_{\ur}(C)\,;\,\alpha\mapsto \alpha|_C
\]
of the unramified cohomology groups is well-defined. Hence, Definition \ref{def:ct, ur coh} makes sense. 
\end{rem}

By definition, the unramified curve-tame cohomology has full functoriality as follows.

\begin{prop}(Functoriality Property)\label{prop:t, ur coh cntrv}
Let $k$ be a field of characteristic $p>0$ and $K/k$ a finitely generated field extension. Let $X$ be a regular variety over $k$ and $Y$ a regular variety over $K$. Let $f\colon Y\to X$ be a morphism over $k$, where $Y$ is viewed as a $k$-scheme via the composition $Y\to\Spec K\to\Spec k$. Then the natural restriction map $f^*\colon H^{i+1,i}_{\ur}(X)\to H^{i+1,i}_{\ur}(Y)$ induces the map $H^{i+1,i}_{\ct,\ur}(X/k)\to H^{i+1,i}_{\ct,\ur}(Y/K)$ between the unramified curve-tame cohomology groups. 
\end{prop}

\begin{proof}
This is immediate from Definition \ref{def:ct, ur coh}.
\end{proof}

For later use, we collect some basic facts on the tame subgroups in the next proposition.

\begin{prop}\label{prop:t, ur coh}
Let $k$ be a field of characteristic $p>0$, $X$ a regular variety over $k$ and $i\ge 0$ an integer. 
\begin{enumerate}
\renewcommand{\labelenumi}{(\arabic{enumi})}

\item If $X=C$ is a normal $k$-curve, then we have $H^{i+1,i}_{\tame,\ur}(C/k)=H^{i+1,i}_{\ct,\ur}(C/k)$.

\item If $X$ is a proper smooth geometrically connected variety over $k$, then we have 
\[
H^{i+1,i}_{\ct,\ur}(X/k)=H^{i+1,i}_{\tame,\ur}(X/k)=H^{i+1,i}_{\ur}(X)=H^{i+1,i}_{\ur}(k(X)/k).
\]

\item If $X=\A^1_k$ is the affine line, then there exist natural isomorphisms
\begin{equation*} 
H^{i+1,i}(k)\xrightarrow{\simeq}H^{i+1,i}_{\ur}(\bbP^1_k)\xrightarrow{\simeq}H^{i+1,i}_{\tame,\ur}(\A^1_k/k).
\end{equation*} 

\item More generally, the projection map ${\rm pr}_X\colon X\times\A^1_k\to X$ induces an isomorphism between the unramified curve-tame cohomology groups 
\begin{equation*}
H^{i+1,i}_{\ct,\ur}(X/k)\xrightarrow{\simeq}H^{i+1,i}_{\ct,\ur}(X\times\A^1_k/k).
\end{equation*} 
\end{enumerate}
\end{prop}

\begin{proof}
(1) The inclusion $H^{i+1,i}_{\ct,\ur}(C/k)\subseteq H^{i+1,i}_{\tame,\ur}(C/k)$ is obvious. 
Let us show the inclusion $H^{i+1,i}_{\tame,\ur}(C/k)\subseteq H^{i+1,i}_{\ct,\ur}(C/k)$. Let $\alpha\in H^{i+1,i}_{\tame,\ur}(C/k)$ be an arbitrary element. Let $K/k$ be a finitely generated field extension and $D$ a normal $K$-curve. Suppose given a $k$-morphism $f\colon D\to C$. We have to show that $f^*\alpha\in H^{i+1,i}_{\ur}(D)$ belongs to the tame subgroup $H^{i+1,i}_{\tame,\ur}(D/K)$. Let $E\subset C$ be the closure of the image $f(D)$ in $C$. Then $E=C$ or $E=\{x\}$ for some closed point $x\in C$. In the former case, the morphism $f\colon D\to C$ is dominant, hence again by Remark \ref{rem:tame, ur dominant map}, one can conclude that $f^*\alpha\in H^{i+1,i}_{\tame,\ur}(D/k)$. Let us suppose that $E=\{x\}$ for some $x\in C_{(0)}$. In this case, it suffices to prove that the map $H^{i+1,i}(k(x))\to H^{i+1,i}_{\ur}(D)$ factors through $H^{i+1,i}_{\tame,\ur}(D/K)$. However, as the extension $k(x)/k$ is finite, for any $v\in\DIV(K(D)/K)$, the restriction $v|_{k(x)}$ is the trivial valuation. Hence, the image of the map $H^{i+1,i}(k(x))\to H^{i+1,i}_{\ur}(D)$ is contained in $H^{i+1,i}_{\ur}(K(D)/K)\subset H^{i+1,i}_{\tame, \ur}(D/K)$. This completes the proof.

(2) The last equality is due to Proposition \ref{prop:ur coh cmp}. The equality $H^{i+1,i}_{\tame,\ur}(X/k)=H^{i+1,i}_{\ur}(X)$ follows from the fact that $H^{i+1,i}_{\ur}(X)=H^{i+1,i}_{\ur}(k(X)/k)\subseteq H^{i+1,i}_{\tame}(k(X)/k)$. Let us prove the equality $H^{i+1,i}_{\ct,\ur}(X/k)=H^{i+1,i}_{\ur}(X)$. Let $\alpha\in H^{i+1,i}_{\ur}(X)$ be an arbitrary element. Let $K/k$ be a finitely generated field extension and $C$ a normal $K$-curve. Suppose given a $k$-morphism $f\colon C\to X$. We have to show that $f^*\alpha\in H^{i+1,i}_{\ur}(C)$ belongs to $H^{i+1,i}_{\tame,\ur}(C/K)$. However, as the morphism $f\colon C\to X$ factors through the base change $X_K=X\times_k K$. As $X_K$ is proper over $K$, the $K$-morphism $C\to X_K $ can be extended to a morphism $\overline{C}\to X$ from the normal compactification $\overline{C}$ of $C$. Therefore, $f^*(\alpha)$ belongs to $H^{i+1,i}_{\ur}(\overline{C}/K)\subset H^{i+1,i}_{\tame,\ur}(C/K)$. This completes the proof. 

(3) This follows from Remark \ref{rem:H, RC}\,(1). 

(4) Let $s\colon X\to X\times\A^1_k$ be the map $x\mapsto (x,0)$ and $s_{\eta}\colon \Spec k(X)\to\A^1_{k(X)}$ the natural base change of $s$ along the generic point $\eta\colon \Spec k(X)\to X$. By Proposition \ref{prop:t, ur coh cntrv}, we have a commutative diagram
\begin{equation*}
\begin{xy}
\xymatrix{
H^{i+1,i}(k(X))& H^{i+1,i}_{\ct,\ur}(\A^1_{k(X)}/k(X))\ar[l]_{s_{\eta}^*~~~}^{\simeq~~~}\\
H^{i+1,i}_{\ct,\ur}(X/k)\ar@{^{(}->}[u]& H^{i+1,i}_{\ct,\ur}(X\times\A^1_k/k)\ar@{^{(}->}[u]\ar[l]^{s^*~~},
}
\end{xy}
\end{equation*}
where the two vertical arrows are injective by definition of the unramified tame cohomology. Moreover, the top horizontal arrow is an isomorphism, which is due to (3). This implies that $s^*$ is injective. However, as ${\rm pr}_X\circ s={\rm id}_X$, we have $s^*\circ {\rm pr}_X^*={\rm id}$ on $H^{i+1,i}_{\ct,\ur}(X/k)$, which implies that $s^*$ is surjective. Therefore, $s^*$ is bijective and its inverse map is given by ${\rm pr}_X^*\colon H^{i+1,i}_{\ct,\ur}(X/k)\to H^{i+1,i}_{\ct,\ur}(X\times\A^1_k/k)$. This completes the proof. 
\end{proof}

Finally, we discuss corestriction maps on the tame cohomology groups. Let us begin with the local situation.

\begin{lem}\label{lem:def of Cor^t cdvf}
Let $L/K$ be a finite extension of complete discrete valuation fields of characteristic $p>0$. Then the corestriction map $\Cor_{L/K}\colon H^{i+1,i}(L)\to H^{i+1,i}(K)$ induces a map $H^{i+1,i}_{\tame}(L)\to H^{i+1,i}_{\tame}(K)$ between the tame cohomology groups, which we denote by $\Cor^{\rmt}_{L/K}$.  
\end{lem}

\begin{proof}
First let us consider the case where $L/K$ is unramified. In this case we have $L^{ur}=K^{ur}$ and $L\otimes_K K^{ur}=\prod_{i=1}^{[L:K]}L^{ur}$. Thus by the equation (C\ref{eq:Res vs Cor}), we have $\Res_{K^{ur}/K}\circ\Cor_{L/K}=[L:K]\Res_{L^{ur}/L}$, which immediately implies that the composition $H^{i+1,i}_{\tame}(L)\subset H^{i+1,i}(L)\xrightarrow{\Cor_{L/K}} H^{i+1,i}(L)$ factors through the tame cohomology group $H^{i+1,i}_{\tame}(K)$. Therefore, the lemma is true if $L/K$ is an unramified extension. Then thanks to the transitivity of the corestriction maps~(C\ref{eq:Cor trans}), it remains to prove the assertion for a finite extension $L/K$ of complete discrete valuation fields whose residue extension is purely inseparable. In this case we have $\Gal(L^{ur}/L)\xrightarrow{\simeq}\Gal(K^{ur}/K)$ and $L\otimes_K K^{ur}\simeq LK^{ur}$, hence $L\otimes_K K^{ur}\simeq L^{ur}$. Therefore, again by (C\ref{eq:Res vs Cor}), we can conclude that the map $\Cor_{L/K}\colon H^{i+1,i}(L)\to H^{i+1,i}(K)$ induces a map $H^{i+1,i}_{\tame}(L)\to H^{i+1,i}_{\tame}(K)$. This completes the proof.      
\end{proof}

For any $v\in\DIV(K/k)$, we have $L\otimes_K K_v\simeq\prod_{w\mid v}L_w$~(cf.\ Remark \ref{rem:div=>excellent}). Therefore, the equation (C\ref{eq:Res vs Cor}) together with Lemma \ref{lem:def of Cor^t cdvf}  implies that the diagram
\begin{equation*}
\begin{xy}
\xymatrix{
H^{i+1,i}(L)\ar[r]\ar[d]_{\Cor_{L/K}}&\bigoplus_{w\mid v}H^{i+1,i}(L_w)\ar[d]^{\sum_{w\mid v}\Cor_{L_w/K_v}}&\bigoplus_{w\mid v}H^{i+1,i}_{\tame}(L_w)\ar[d]^{\sum_{w\mid v}\Cor^{\rmt}_{L_w/K_v}}\ar@{_{(}->}[l]\\
H^{i+1,i}(K)\ar[r]&H^{i+1,i}(K_v)&H^{i+1,i}_{\tame}(K_v)\ar@{_{(}->}[l].
}
\end{xy}
\end{equation*}
is commutative. As a consequence, we obtain the corestriction map of the tame cohomology groups,
\begin{equation}\label{eq:Cor^t}
\Cor^{\rmt}_{L/K}\colon H^{i+1,i}_{\tame}(L/k)\to H^{i+1,i}_{\tame}(K/k). 
\end{equation}

Let us prove the following.

\begin{prop}\label{prop:Cor on t, ur coh}
Let $k$ be a field of characteristic $p>0$ and $f\colon Y\to X$ a finite surjective morphism of normal varieties over $k$. Then the corestriction map $\Cor^{\rmt}_{k(Y)/k(X)}\colon H^{i+1,i}_{\tame}(k(Y)/k)\to H^{i+1,i}_{\tame}(k(X)/k)$ of the tame cohomology groups induces the map
\begin{equation*}
\Cor_{Y/X}^{\rmt}\colon H^{i+1,i}_{\tame,\ur}(Y/k)\to H^{i+1,i}_{\tame,\ur}(X/k)
\end{equation*}
between the na\"ive unramified tame cohomology groups. \end{prop}

For the proof, we need two lemmas.

\begin{lem}\label{lem:Cor^t cdvf}
Let $K$ be a complete discrete valuation field with $v$ its valuation. Let $L/K$ be a finite extension and $w$ the unique extension of $v$ to $L$. Then  for any $i\ge 0$, we have a commutative diagram
\begin{equation*}
\begin{xy}
\xymatrix{
H^{i+1,i}_{\tame}(L)\ar[r]^{\partial_w~~}\ar[d]^{\Cor_{L/K}^{\rmt}}&H^{i,i-1}(k(w))\ar[d]^{\Cor_{k(w)/k(v)}}\\
H^{i+1,i}_{\tame}(K)\ar[r]^{\partial_v~~}& H^{i,i-1}(k(v)).
}
\end{xy}
\end{equation*} 
\end{lem}

\begin{proof}
We first prove the commutativity of the diagram  under the assumption that $k(w)/k(v)$ is purely inseparable. Let $[f,g_1,\dots,g_i\}\in H^{i+1,i}_{\tame}(L)$ be an arbitrary generator. As $k(w)/k(v)$ is purely inseparable and both the valuation rings $\sO_v$ and $\sO_w$ are complete, the restriction map $H^{1,0}(\sO_v)\to H^{1,0}(\sO_w)$ is an isomorphism. Therefore, we may assume that $f=f'_L$ for some element $f'\in \sO_v$. In this case, by  (C\ref{eq:proj formula 2}) and the equation (\ref{eq:partial_v vs partial_v^M}) together with the compatibility between the norm $N$ and the residue $\partial^{\rm M}$ for the Milnor K-groups~(cf.\ \cite[Proposition 7.4.1]{GS06}), we have
\begin{equation*}
\begin{aligned}
\partial_{v}(\Cor^{\rmt}_{L/K}([f,g_1,\dots,g_i\}))
&=\partial_v([f',N_{L/K}(\{g_1,\dots,g_i\}))\\
&=[\overline{f}',\partial_v^{\rm M}(N_{L/K}(\{g_1,\dots,g_i\}))\\
&=[\overline{f}',N_{k(w)/k(v)}(\partial_{w}^{\rm M}\{g_1,\dots,g_i\}))\}\\
&=\Cor_{k(w)/k(v)}([\overline{f},\partial_w^{\rm M}(\{g_1,\dots,g_i\})\})\\
&=\Cor_{k(w)/k(v)}(\partial_w([f,g_1,\dots,g_i\})).
\end{aligned}
\end{equation*}
This proves the lemma in the case where $k(w)/k(v)$ is a purely inseparable extension.  

By the transitivity of the corestriction map~ (C\ref{eq:Cor trans}), it remains to prove the commutativity of the diagram in the case where $w/v$ is unramified, or equivalently the extension of the valuation rings $\sO_v\to\sO_w$ is \'etale. In this case, we have $\Omega^i_{\sO_w}\simeq\Omega^i_{\sO_v}\otimes_{\sO_v}\sO_w$. Hence, the group $H^{i+1,i}_{\tame}(L)$ is generated by elements of the form $[f,\omega_L\}$ where $f\in\sO_w$ and $\omega=\dfrac{dg_1}{g_1}\wedge\cdots\wedge\dfrac{dg_i}{g_i}\in \Omega^{i}_{K,\log}$. Therefore, by (C\ref{eq:proj formula 1}), we have 
\begin{equation*}
\Cor_{L/K}^{\rmt}([f,\omega_L\})=[\Cor^{\rmt}_{L/K}(f),\omega\}. 
\end{equation*}
Note that $\Cor^{\rmt}_{L/K}(f)\in H^{1,0}(K)=H^{1,0}(\sO_v)$. If $v(g_j)=0$ for $1\le j\le i$, then we have
\begin{equation*}
\partial_v(\Cor_{L/K}^{\rmt}([f,\omega_L\}))=0=\Cor_{k(w)/k(v)}(\partial_w([f,\omega_L\})).
\end{equation*}
On the other hand, if $v(g_1)=w((g_1)_L)=1$ and $v(g_j)=0$ for $j\ge 2$, we have
\begin{equation*}
\begin{aligned}
\partial_v(\Cor_{L/K}^{\rmt}([f,\omega_L\}))&=[\overline{\Cor^{\rmt}_{L/K}(f)},\overline{g}_2,\dots,\overline{g}_i\},\\
\Cor_{k(w)/k(v)}(\partial_w([f,\omega_L\}))&=\Cor_{k(w)/k(v)}([\overline{f},(\overline{g}_2)_{k(w)},\dots,(\overline{g}_i)_{k(w)}\})\\
&=[\Cor_{k(w)/k(v)}(\overline{f}),\overline{g}_2,\dots,\overline{g}_i\}.
\end{aligned}
\end{equation*}
Therefore, it suffices to show that $\overline{\Cor^{\rmt}_{L/K}(f)}=\Cor_{k(w)/k(v)}(\overline{f})$ in $H^{1,0}(k(v))$, or equivalently to show that the diagram
\begin{equation}\label{eq:lem:Cor^t cdvf}
\vcenter{
\xymatrix{
H^{1,0}(L)\ar[d]_{\Cor_{L/K}}&H^{1,0}(\sO_w)\ar@{_{(}->}[l]\ar[r]^{\simeq}&H^{1,0}(k(w))\ar[d]^{\Cor_{k(w)/k(v)}}\\
H^{1,0}(K)&H^{1,0}(\sO_v)\ar@{_{(}->}[l]\ar[r]^{\simeq}&H^{1,0}(k(v))
}
}
\end{equation}
is commutative. However, as both the extensions $L/K$ and $k(w)/k(v)$ are separable, by \cite[Remark 5.1(3)]{Shiho07}, the corestriction maps $\Cor_{L/K}$ and $\Cor_{k(w)/k(v)}$ are induced by the trace maps ${\rm Tr}_{L/K}\colon L\to K$ and ${\rm Tr}_{k(w)/k(v)}\colon k(w)\to k(v)$ respectively. Now the commutativity of the diagram (\ref{eq:lem:Cor^t cdvf}) follows from the \'etaleness of the extension $\sO_v\to\sO_w$. This completes the proof of the lemma. 
\end{proof}

\begin{lem}\label{lem:Cor on t, ur coh}
Let $L/K$ be a finite extension of finitely generated fields over $k$. Let $v\in \DIV(K/k)$.  Then the diagram
\begin{equation*}
\begin{xy}
\xymatrix{
H^{i+1,i}_{\tame}(L/k)\ar[r]^{(\partial_w)\quad\quad}\ar[d]_{\Cor_{L/K}^{\rmt}}& \bigoplus_{w\mid v}H^{i,i-1}(k(w))\ar[d]^{\sum_{w\mid v}\Cor_{k(w)/k(v)}}\\
H^{i+1,i}_{\tame}(K/k)\ar[r]^{\partial_v}& H^{i,i-1}(k(v))
}
\end{xy}
\end{equation*} 
is commutative, where $w$ is taken over all the valuations on $L$ lying above $v$. 
\end{lem}

\begin{proof}
By definition of the corestriction map for the unramified tame cohomology, the diagram decomposes into the two squares
\begin{equation*}
\begin{xy}
\xymatrix{
H^{i+1,i}_{\tame}(L/k)\ar[r]\ar[d]_{\Cor^{\rmt}_{L/K}}&\bigoplus_{w\mid v}H^{i+1,i}_{\tame}(L_w)\ar[r]^{(\partial_w)~~~}\ar[d]_{\sum_{w\mid v}\Cor_{L_w/K_v}^{\rmt}}& \bigoplus_{w\mid v}H^{i,i-1}(k(w))\ar[d]^{\sum_{w\mid v}\Cor_{k(w)/k(v)}}\\
H^{i+1,i}_{\tame}(K/k)\ar[r]&H^{i+1,i}_{\tame}(K_v)\ar[r]^{\partial_v}& H^{i,i-1}(k(v)).
}
\end{xy}
\end{equation*} 
The left square is commutative by the construction of the corestriction map $\Cor^{\rmt}_{L/K}$. The second one is also commutative, which follows from  Lemma \ref{lem:Cor^t cdvf}. This completes the proof. 
\end{proof}

\begin{proof}[Proof of Proposition \ref{prop:Cor on t, ur coh}]
By Lemma \ref{lem:Cor on t, ur coh} together with the transitivity (C\ref{eq:Cor trans}) of the corestriction maps, we have a commutative diagram
\begin{equation*}
\begin{xy}
\xymatrix{
H^{i+1,i}_{\tame}(k(Y)/k)\ar[r]^{(\partial_y)_y\quad}\ar[d]_{\Cor^{\rmt}_{k(Y)/k(X)}}&\bigoplus_{y\in Y^{(1)}}H^{i,i-1}(k(y))~~~~~~\ar[d]^{\left(\sum_{y\mid x}\Cor_{k(y)/k(z)}\right)_{x}}\\
H^{i+1,i}_{\tame}(k(X)/k)\ar[r]^{(\partial_{x})_x\quad}&\bigoplus_{x\in X^{(1)}}H^{i,i-1}(k(x)),
}
\end{xy}
\end{equation*}
which immediately implies the assertion. This completes the proof. 
\end{proof}


\section{A pairing with the Suslin homology and its application}\label{sec:pairing}

Let $k$ be a field of characteristic $p\ge 0$ and $X$ a smooth geometrically connected variety over $k$. In \cite{Kahn11}, by extending Merkurjev's corepresentability theorem\cite[Theorem 2.10]{Mer08}, Bruno Kahn proved that there exists a cycle module $H^X_*$ over $k$ together with a natural isomorphism
\[
A^0(X,M_0)\simeq\Hom_{\bf CM}(H_*^X,M_*)
\]
for any cycle module $M_*$ over $k$\,(see \cite[Theorem 1.4]{Kahn11}), where $A^0(X,M_0)$ is the Chow group with coefficients in $M_0$ in the sense of Rost\cite[\S5]{Rost} and $\bf CM$ is the category of cycle modules over $k$. The corepresenting object $H^X_*$  he constructed is closely related with the $0$-th Suslin homology group $H^S_0(X)$. Namely, for any finitely generated field extension $K/k$, we have
\[
H_0^X(K)\simeq H^S_0(X_K).
\]
Moreover, the degree map $\Deg\colon H^S_0(X_K)\to\Z$ is an isomorphism for any such $K/k$ if and only if the natural projection map $H_*^X\to K^{\rm M}_*$ into the Milnor K-theory is an isomorphism. As a consequence (see \cite[Corollary 4.7]{Kahn11}), he also established a refinement of Merkurjev's characterization of universal triviality of the Chow group of zero-cycles \cite[Theorem 2.11]{Mer08}. In particular, he proved that if the degree map $\Deg\colon H^S_0(X_K)\to\Z$ is an isomorphism for any finitely generated field extension $K/k$, we have $M_0(k)\xrightarrow{\simeq}A^0(X,M_0)$ for any cycle module $M_*$ over $k$. By applying the result to the cycle module associated with \'etale cohomology $H^{i,j}(-)=H^{i}(-,\Z/\ell^n\Z(j))$~(cf.\ Remark \ref{rem:A^0}), for $X$ having universally trivial Suslin homology $H^S_0$, we obtain the triviality of the unramified cohomology, i.e.
\begin{equation}\label{eq:triv ur coh}
H^{i,j}(k)\xrightarrow{\simeq}H^{i,j}_{\ur}(X)
\end{equation}
unless $H^{i,j}(-)=H^{i}(-,\Z/p^n\Z(i-1))$ with $\ell=p$.

In this section, we investigate the exceptional case, i.e.\ the case where $\ell=p$ with $H^{i+1,i}(-)=H^{i+1}(-,\Z/p^n\Z(i))=H^1_{\et}(-,W_n\Omega^{i}_{\log})$, and prove a mod $p$ analogue to the above triviality result for unramified cohomology\,(see Corollary \ref{cor:pairing}). 
Let us assume that $k$ is a field of characteristic $p>0$ and consider the mod $p$ unramified cohomology
\[
H^{i+1,i}_{\ur}(X)=H^{i+1,i}_{\ur}(X,\Z/p\Z(i))=H^{1,i}_{\ur}(X,\Omega_{X,\log}^i).
\] 
We continue to use the same notation as in the previous section. As noticed in the first paragraph in the previous section, the same triviality as (\ref{eq:triv ur coh}) does not hold for the whole mod $p$ unramified cohomology. As a modification, we first replace the unramified cohomology $H^{i+1,i}_{\ur}(X)$ by its curve-tame subgroup $H^{i+1,i}_{\ct,\ur}(X/k)$. However, as it is unclear if the curve-tame subgroup can be described as the Chow group $A^0(X,M_0)$ for some cycle module $M_*$, it does not seem straightforward to apply the previous result \cite[Corollary 4.7]{Kahn11}. Instead, we follow the argument of Auel et al.\ in \cite[\S3]{ABBB} to recover a part of the method in \cite{Kahn11}. Namely, we construct a pairing of the form
\begin{equation*}
H^S_0(X)\times H^{i+1,i}_{\ct,\ur}(X/k)\to H^{i+1,i}(k)
\end{equation*}
which fulfills satisfactory conditions (see Theorem \ref{thm:pairing}). 

First of all, recall that
\begin{equation*}
H^S_0(X)\Def\Cok\bigl(c(\A^1_k,X)\xrightarrow{s_0^*-s_1^*}Z_0(X)\bigl),
\end{equation*}
where $c(\A^1_k,X)$ denotes the free abelian group generated by integral closed subschemes $\Gamma$ of $\A^1_k\times X$ such that the composite $\Gamma\subset \A^1_k\times X\to\A^1_k$ is a finite surjective morphism, and $s_0,s_1\colon\Spec k\to\A^1_k$ are the sections corresponding to the points $0,1\in\A^1_k$ respectively.

We begin by the pairing with the group $Z_0(X)$ of $0$-cycles,
\begin{equation}\label{eq:pairing with Z_0}
\langle-,-\rangle\colon Z_0(X)\times H^{i+1,i}_{\ct,\ur}(X/k)\to H^{i+1,i}(k),
\end{equation}
which is defined by the composite
\begin{equation*}
\begin{aligned}
Z_0(X)\times H^{i+1,i}_{\ct,\ur}(X/k)\subset& Z_0(X)\times H^{i+1,i}_{\ur}(X)\\
\xrightarrow{(x^*)_{x\in X_{(0)}}}&\bigoplus_{x\in X_{(0)}}H^{i+1,i}(k(x))\xrightarrow{\sum_{x\in X_{(0)}}\Cor_{k(x)/k}} H^{i+1,i}(k). 
\end{aligned}
\end{equation*}
Here, for each closed point $x\in X_{(0)}$, thanks to Corollary \ref{cor:ur coh cmp}\,(2), the restriction map
\[
x^*\colon H^{i+1,i}_{\ur}(X)\to H^{i+1,i}_{\ur}(\Spec k(x))=H^{i+1,i}(k(x))
\]
is well-defined. For any finitely generated field extension $K/k$, this construction similarly gives a pairing
\begin{equation*}
Z_0(X_K)\times H^{i+1,i}_{\ct,\ur}(X_K/K)\to H^{i+1,i}(K).
\end{equation*}
Then the pairing satisfies the following partial compatibility conditions, which are enough for our later application.


\begin{lem}\label{lem:pairing with Z_0}
Let $k,X$ be as above. 
Let $\langle-,-\rangle$ be the pairing given in (\ref{eq:pairing with Z_0}). 
\begin{enumerate}
\renewcommand{\labelenumi}{(\arabic{enumi})}
\item Let $K\Def k(X)$ be the function field. Let $\eta\in (X_{K})_{(0)}$ be the closed point associated with the generic point of $X$. Then the composite
\begin{equation*}
H^{i+1,i}_{\ct,\ur}(X/k)\to H^{i+1,i}_{\ct,\ur}(X_{K}/K)\xrightarrow{\langle\eta,-\rangle}H^{i+1,i}(K)
\end{equation*}
gives the natural inclusion map $H^{i+1,i}_{\ct,\ur}(X/k)\subset H^{i+1,i}(K)$. 

\item Let $x\in X_{(0)}$ be a closed point whose residue extension $k(x)/k$ is separable. Let $\alpha\in H^{i+1,i}_{\ct,\ur}(X/k)$ be an arbitrary element. Then we have $\langle x_{K},\alpha_{K}\rangle=\langle x,\alpha\rangle_K$ in $H^{i+1,i}(K)$ for any finitely generated field extension $K/k$. 
\end{enumerate}
\end{lem}

\begin{proof}
(1) By Proposition \ref{prop:t, ur coh cntrv}, we have the commutative diagram
\begin{equation*}
\begin{xy}
\xymatrix{
H^{i+1,i}_{\ct,\ur}(X/k)\ar[r]\ar@{^{(}->}[d]&H^{i+1,i}_{\ct,\ur}(X_K/K)\ar@{^{(}->}[d]\\
H^{i+1,i}_{\ur}(X)\ar@{_{(}->}[rd]\ar[r]&H^{i+1,i}_{\ur}(X_K)\ar[d]^{\eta^*}\\
&H^{i+1,i}(K).
}
\end{xy}
\end{equation*}
Thus, the assertion holds.

(2) Let $f\colon X_K\to X$ be the natural projection map. Then the fiber is given by $f^{-1}(x)=\Spec (K\otimes_kk(x))$. As $k(x)/k$ is a finite separable extension, the fiber $f^{-1}(x)$ is reduced and hence $x_K=\sum_{y\mid x}y$ in $Z_0(X_K)$. Therefore, by the equation (C\ref{eq:Res vs Cor}), we have
\begin{equation*}
\begin{aligned}
\langle x,\alpha\rangle_K&=\Res_{K/k}\circ\Cor_{k(x)/k}(x^*\alpha)\\
&=\sum_{y\mid x}\Cor_{K(y)/K}\circ\Res_{K(y)/k(x)}(x^*(\alpha))\\
&=\sum_{y\mid x}\Cor_{K(y)/K}(y^*(\alpha_K))=\langle x_K,\alpha_K\rangle. 
\end{aligned}
\end{equation*}
This completes the proof. 
\end{proof}

Now we prove the key result. 

\begin{thm}\label{thm:pairing}
Let $X$ be a smooth geometrically connected  variety over a field $k$ of characteristic $p>0$. Let $i\ge 0$ be an integer. For any finitely generated field extension $K/k$, there exists a pairing
\begin{equation*}
H^S_0(X_K)\times H^{i+1,i}_{\ct,\ur}(X_K/K)\to H^{i+1,i}(K)
\end{equation*} 
which satisfies the following conditions.
\begin{enumerate}
\renewcommand{\labelenumi}{(\arabic{enumi})}
\item If $\eta\in (X_{k(X)})_{(0)}$ is the closed point associated with the generic point $\Spec k(X)\to X$, then the composite $H^{i+1,i}_{\ct,\ur}(X/k)\to H^{i+1,i}_{\ct,\ur}(X_{k(X)}/k(X))\xrightarrow{\langle\eta,-\rangle}H^{i+1,i}(k(X))$ 
is given by the natural inclusion map $H^{i+1,i}_{\ct,\ur}(X/k)\subset H^{i+1,i}(k(X))$. 

\item If $x\in X_{(0)}$ is a closed point such that the residue extension $k(x)/k$ is separable, then for any element $\alpha\in H^{i+1,i}_{\ct,\ur}(X/k)$ and for any finitely generated field extension $K/k$, we have $\langle x_{K},\alpha_{K}\rangle=\langle x,\alpha\rangle_K$ in $H^{i+1,i}(K)$. 
\end{enumerate}
\end{thm}

For the proof, we need two lemmas.

\begin{lem}\label{lem:pairing 1}
With the same notation as in Theorem \ref{thm:pairing}, let $\phi\colon C\to X$ be a morphism from a normal $k$-curve $C$ into $X$. Then for any $0$-cycle $z\in Z_0(C)$ and any element $\alpha\in H^{i+1,i}_{\ct,\ur}(X/k)$, we have
\begin{equation*}
\langle\phi_*(z),\alpha\rangle=\langle z,\phi^*(\alpha)\rangle, 
\end{equation*}
where $\langle-,-\rangle$ is the pairing given by (\ref{eq:pairing with Z_0}).
\end{lem}

\begin{proof}
We adapt the same argument as in the proof of \cite[Lemma 3.2]{ABBB}.
Without loss of generality, we may assume that $z$ is a closed point of $C$. Recall that $\phi_*(z)=[k(z):k(x)] x$, where we put $x\Def\phi(z)$. Since the diagram 
\begin{equation*}
\begin{xy}
\xymatrix{
\Spec k(z)\ar[r]^{~~z}\ar[d]&C\ar[d]^{\phi}\\
\Spec k(x)\ar[r]^{~~x}& X
}
\end{xy}
\end{equation*}
is commutative, by the transitivity of corestriction maps~(C\ref{eq:Cor trans}) together with the equality $\Cor_{k(z)/k(x)}\circ\Res_{k(z)/k(x)}=[k(z):k(x)]$~(cf.\ (\ref{eq:Cor Res =deg})), we have
\begin{equation*}
\begin{aligned}
\langle\phi_*(z),\alpha\rangle
&=\Cor_{k(x)/k}(\Cor_{k(z)/k(x)}(\Res_{k(z)/k(x)}(x^*\alpha)))\\
&=\Cor_{k(z)/k}(\Res_{k(z)/k(x)}(x^*\alpha))\\
&=\Cor_{k(z)/k}(z^*\phi^*\alpha)=\langle z,\phi^*\alpha\rangle.
\end{aligned}
\end{equation*}
This completes the proof of the lemma. 
\end{proof}

\begin{lem}\label{lem:pairing 2}
With the same notation as in Theorem \ref{thm:pairing}, let $\phi\colon \Gamma\to C$ be a finite surjective morphism between normal $k$-curves. Let $x\colon \Spec k(x)\to C$ be a closed point and $\Gamma_x\Def\phi^{-1}(x)$ the scheme-theoretic fiber of $x$. Then for any $\alpha\in H^{i+1,i}_{\tame,\ur}(\Gamma/k)$, we have
\begin{equation*}
\langle\Gamma_x,\alpha\rangle=\langle x,\Cor^{\rmt}_{\Gamma/C}(\alpha)\rangle,
\end{equation*}
where $\langle-,-\rangle$ is the pairing given by (\ref{eq:pairing with Z_0}).\end{lem}

\begin{proof}
Fix a uniformizer $\pi_x\in\sO_{C,x}$ at $x$. Moreover, for each point $y\in\Gamma$ lying above $x$, fix a uniformizer $\pi_y\in\sO_{\Gamma,y}$ and put $e_y\Def v_y(\pi_x)$. By using Lemmas \ref{lem:s_v dvr} and \ref{lem:Cor on t, ur coh} together with the formula (C\ref{eq:proj formula 1}), one can compute
\begin{equation*}
\begin{aligned}
\langle x,\Cor^{\rmt}_{\Gamma/C}(\alpha)\rangle
&=\Cor_{k(x)/k}(x^*\Cor^{\rmt}_{\Gamma/C}(\alpha))\\
&=(-1)^i\Cor_{k(x)/k}(\partial_x([\Cor^{\rmt}_{\Gamma/C}(\alpha),\pi_x\}))\\
&=(-1)^i\Cor_{k(x)/k}(\partial_x(\Cor^{\rmt}_{k(\Gamma)/k(C)}([\alpha,\phi^*(\pi_x)\})))\\
&=(-1)^i\Cor_{k(x)/k}(\sum_{y\mid x}\Cor_{k(y)/k(x)}(\partial_y([\alpha,\phi^*(\pi_x)\})))\\
&=\sum_{y\mid x}e_y\Cor_{k(y)/k}((-1)^i\partial_y([\alpha,\pi_y\}))\\
&=\sum_{y\mid x}e_y\Cor_{k(y)/k}(y^*\alpha)
=\sum_{y\mid x}e_y\langle y,\alpha\rangle
=\langle\Gamma_x,\alpha\rangle. 
\end{aligned}
\end{equation*}
This completes the proof. 
\end{proof}

\begin{proof}[Proof of Theorem \ref{thm:pairing}]
If one proves the pairing (\ref{eq:pairing with Z_0}) factors through $H^S_0(X)\times H^{i+1,i}_{\ct,\ur}(X/k)$, then the required conditions (i) and (ii) are immediate from Lemma \ref{lem:pairing with Z_0}. Let $\alpha\in H^{i+1,i}_{\ct,\ur}(X/k)$ be an arbitrary element. Let $\Gamma\subset\A^1_k\times X$ be an integral closed subscheme such that the projection $\phi\colon\Gamma\to\A^1_k=\Spec k[t]$ is finite and surjective. Let $\Gamma_0\Def\phi^{-1}(0)$ and $\Gamma_1\Def\phi^{-1}(1)$ be the scheme-theoretic fibers of $0$ and $1$ respectively. We have to show that
\begin{equation*}
\langle{\rm pr}_{X*}(\Gamma_0-\Gamma_{1}),\alpha\rangle=0.
\end{equation*} 
First note that 
\begin{equation*}
\Gamma_0-\Gamma_1=\phi^*{\rm div}_{\A^1_k}\Bigl(\frac{t}{t-1}\Bigl)={\rm div}_{\Gamma}\Bigl(\frac{\phi^*(t)}{\phi^*(t)-1}\Bigl)=\pi_*{\rm div}_{\widetilde{\Gamma}}\Bigl(\frac{\phi^*(t)}{\phi^*(t)-1}\Bigl)=\pi_*(\widetilde{\Gamma}_0-\widetilde{\Gamma}_1),
\end{equation*}
where $\pi\colon\widetilde{\Gamma}\to\Gamma$ is the normalization of the curve $\Gamma$. Thus by applying Lemma \ref{lem:pairing 1} to the composition $\widetilde{\Gamma}\xrightarrow{\pi}\Gamma\xrightarrow{{\rm pr}_X}X$, we are reduced to prove that for any finite surjective morphism $\phi\colon C\to\A^1_k$ from a normal connected $k$-curve and any element $\alpha\in H^{i+1,i}_{\tame,\ur}(C/k)$, we have $\langle C_0,\alpha\rangle=\langle C_1,\alpha\rangle$. However, according to Lemma \ref{lem:pairing 2}, we have
\begin{equation*}
\langle C_i,\alpha\rangle=s_i^*\Cor^{\rmt}_{C/\A^1_k}(\alpha)
\end{equation*} 
for $i=0,1$. On the other hand, by Proposition \ref{prop:t, ur coh}(3), we have the natural isomorphism $H^{i+1,i}(k)\xrightarrow{\simeq}H^{i+1,i}_{\tame,\ur}(\A^1_k/k)$. Therefore, $s_0^*=s_1^*$ on $H^{i+1,i}_{\tame,\ur}(\A^1_k/k)$, which implies the desired equality $\langle C_0,\alpha\rangle=\langle C_1,\alpha\rangle$. This completes the proof.   
\end{proof}

As an application  of the theorem, we have the following.

\begin{cor}\label{cor:pairing}
Let $X$ be a smooth geometrically connected variety  over a field $k$ of characteristic $p>0$. Suppose that for any finitely generated field extension $K/k$, the degree map induces an isomorphism $\Deg\colon H^S_0(X_K)\xrightarrow{\simeq}\Z$. Then for any integer $i\ge 0$, we have the natural isomorphism $H^{i+1,i}(k)\xrightarrow{\simeq}H^{i+1,i}_{\ct,\ur}(X/k)$. 
\end{cor}

\begin{proof}
Let us prove the injectivity of the map $H^{i+1,i}(k)\to H^{i+1,i}_{\ct,\ur}(X/k)$. As the map $\Deg\colon H^S_0(X)\to\Z$ is surjective, there exists a $0$-cycle $z\in Z_0(X)$ such that $\Deg(z)=1$. It suffices to notice that the composition
\begin{equation*}
H^{i+1,i}(k)\to H^{i+1,i}_{\ct,\ur}(X/k)\xrightarrow{\langle z,-\rangle}H^{i+1,i}(k)
\end{equation*}
is the identity map on $H^{i+1,i}(k)$. This follows from the equation (\ref{eq:Cor Res =deg}) together with the condition that $\Deg(z)=1$. Thus, the natural map $H^{i+1,i}(k)\to H^{i+1,i}_{\ct,\ur}(X/k)$ is injective. 

Let us prove the surjectivity of the map $H^{i+1,i}(k)\to H^{i+1,i}_{\ct,\ur}(X/k)$. 
Let $\alpha\in H^{i+1,i}_{\ct,\ur}(X/k)$ be an arbitrary element. It suffices to show that $\alpha$ is in the image of the map $H^{i+1,i}(k)\to H^{i+1,i}_{\ct,\ur}(X/k)$. We follow the argument in \cite[\S4]{ABBB}. Let $\eta\in (X_{k(X)})_{(0)}$ be the generic point of $X$. 
As the degree map $\Deg\colon H^S_0(X)\to\Z$ is surjective, by \cite[Theorem 9.2]{GLL}, there exists a $0$-cycle $z\in Z_0(X)$ such that $\Deg(z)=1$ and $z$ is supported on closed points having separable residue extension. Therefore, Theorem \ref{thm:pairing}(2) implies that   
\begin{equation*}
\langle z_{k(X)},\alpha_{k(X)}\rangle=\langle z,\alpha\rangle_{k(X)}.
\end{equation*}
On the other hand, as $\Deg (z_{k(X)})=\Deg(\eta)$,  the injectivity of the map $\Deg\colon H^S_0(X_{k(X)})\to\Z$ implies that 
\begin{equation*}
\langle z_{k(X)},\alpha_{k(X)}\rangle=\langle \eta,\alpha_{k(X)}\rangle=\alpha,
\end{equation*}
where the last equality follows from Theorem \ref{thm:pairing}(1). Thus $\alpha=\langle z,\alpha\rangle_{k(X)}$ belongs to the image of the natural map $H^{i+1,i}(k)\to H^{i+1,i}_{\ct,\ur}(X/k)$. This completes the proof. 
\end{proof}

By applying Corollary \ref{cor:pairing} to proper smooth varieties, we obtain the following.

\begin{cor}\label{cor:pairing proper}
Let $X$ be a proper smooth variety over a field $k$ of characteristic $p>0$. Suppose that the degree map $\Deg\colon \CH_0(X_K)\to\Z$ is an isomorphism for any field extension $K/k$. Then for any $i\ge 0$, we have a natural isomorphism $H^{i+1}(k,\Z/p\Z(i))\xrightarrow{\simeq} H^{i+1}_{\ur}(k(X)/k,\Z/p\Z(i))$. 
\end{cor}

\begin{proof}
Noticing that for a proper smooth variety over a field $k$, there exists a natural isomorphism $
H_0^S(X_K)\xrightarrow{\simeq}\CH_0(X_K)$
of abelian groups for any field extension $K/k$, the claim immediately follows from Corollary \ref{cor:pairing} together with Proposition \ref{prop:t, ur coh}\,(2).
\end{proof}


\end{document}